\documentclass[11pt]{amsart}

\usepackage{amssymb}
\usepackage{lscape}

\usepackage{mathpazo}

\usepackage{mathrsfs}

\newcommand{\Vx}{%
\ooalign{\Large $\bigcup$\cr\hss\raisebox{-0.05ex}{\scriptsize $\cdot$}\hss}}
\usepackage[all]{xy}
\usepackage[headings]{fullpage}
\usepackage{paralist}

\usepackage{tabularx,ragged2e,booktabs,caption}
\usepackage[colorlinks=true,linkcolor=blue, citecolor=blue, %
                urlcolor=blue, pagebackref=true]{hyperref}

\makeatletter

\newcommand{\arxiv}[1]{\href{http://arxiv.org/abs/#1}{{\tt arXiv:#1}}}

\newcounter{maintheorem}[equation]
\def\themaintheorem{\thesection.\@arabic \c@maintheorem}
\@addtoreset{equation}{section}
\def\theequation{\thesection.\@arabic \c@equation}

\def\theenumi{\@alph\c@enumi}
\def\theenumii{\@roman\c@enumii}

\makeatother

\theoremstyle{plain}
\newtheorem{theorem}[equation]{Theorem}
\newtheorem*{theorem*}{Theorem}

\newtheorem{corollary}[equation]{Corollary}
\newtheorem{proposition}[equation]{Proposition}

\theoremstyle{definition}

\newtheorem{innercustomthm}{Theorem}
\newenvironment{customthm}[1]
  {\renewcommand\theinnercustomthm{#1}\innercustomthm}
  {\endinnercustomthm}

\newtheorem{remark}[equation]{Remark}

\newtheorem{example}[equation]{Example}

\newtheorem{definition}[equation]{Definition}

\newtheorem{notation}[equation]{Notation}

\newtheorem{discussion}[equation]{Discussion}

\newtheorem{observation}[equation]{Observation}

\newtheorem{construction}[equation]{Construction}

\newcommand{\ui}{\underline{i}}

\newcommand{\m}{{\mathfrak m}}

\def\to{\longrightarrow}

\DeclareMathOperator{\ann}{Ann}

\DeclareMathOperator{\Hom}{Hom}

\DeclareMathOperator{\Hilb}{Hilb}

\DeclareMathOperator{\socle}{Soc}

\makeatletter
\def\RDerChar{\mathbf{R}}
\def\RDer{\@ifnextchar[{\R@Der}{\ensuremath{\RDerChar}}}
\def\R@Der[#1]{\ensuremath{\RDerChar^{#1}}}
\makeatother
\usepackage[normalem]{ulem}

\begin{document}
\begin{abstract}
 In this paper we study  the O-sequences of the local (or graded) $K$-algebras of socle degree $4.$ More precisely, we prove that an O-sequence $h=(1, 3, h_2, h_3, h_4)$, where $h_4 \geq 2,$ is the $h$-vector of a local level $K$-algebra if and only if $h_3\leq 3 h_4.$ We also prove that $h=(1, 3, h_2, h_3, 1)$  is the $h$-vector of a local Gorenstein $K$-algebra  if and only if $h_3 \leq 3$ and $h_2 \leq \binom{h_3+1}{2}+(3-h_3).$ In each of these cases we give an effective method to construct a local level $K$-algebra with a given $h$-vector. Moreover we  refine a result by Elias and Rossi by showing  that if $h=(1,h_1, h_2, h_3, 1)$ is an unimodal Gorenstein O-sequence, then $h$ forces the corresponding   Gorenstein  $K$-algebra to be canonically graded if and only if $h_1=h_3 $ and $h_2=\binom{h_1+1}{2}, $  that is the $h$-vector is maximal.
\end{abstract}

\title[]
{Artinian level algebras of socle degree 4}

\author[Masuti] {Shreedevi K. Masuti}
\address{Dipartimento di Matematica, Universit{\`a} di Genova, Via Dodecaneso 35, 16146 Genova, Italy}
\email{masuti@dima.unige.it}

\author[Rossi]{M.~E.~Rossi}
\address{Dipartimento di Matematica, Universit{\`a} di Genova, Via Dodecaneso 35, 16146 Genova, Italy}
\email{rossim@dima.unige.it}

\date{\today}
\thanks{The first author was supported by INdAM COFOUND Fellowships cofounded by Marie Curie actions, Italy. The second author  was partially supported by PRIN 2014
``Geometria delle varieta' algebriche"}

\maketitle
\section{Introduction}
Let $(A, \m)  $ be an Artinian local or graded  $K$-algebra  where $K$ is  an algebraically closed field of characteristic zero. Let $\socle(A) = (0 : \m) $  be the  socle of $A.$ We denote by $s$ the {\it{socle degree}}  of $A, $ that is the maximum integer $j$ such that $\m^j \neq 0.$ The {\it{type}} of $A$ is $\tau := \dim_K \socle(A). $ Recall that $A$ is said to be {\it level} of type $\tau$ if $\socle(A)=\m^s$ and $\dim_K \m^s=\tau.$ If $A$ has type $1,$ equivalently $ \dim_K  \socle(A) =1, $ then $A$ is {\it Gorenstein}. In the literature   local rings with low socle degree, also called {\it{short local rings,}} have emerged as a testing ground for properties of infinite free resolutions (see \cite{ais08}, \cite{ahs}, \cite{hs11}, \cite{sj79}, \cite{roos05}, \cite{CRV}). They have been also extensively studied  in problems related to the irreducibility and the smoothness of the punctual Hilbert scheme $\Hilb_d(\mathbb{P}^n_K) $ parameterizing zero-dimensional subschemes $\mathbb{P}^n_K$ of degree $d, 
$  see among 
others \cite{
poonen}, \cite{ev}, \cite{cevv}, \cite{cern}.   In this paper we study the structure of level  $K$-algebras of socle degree $4, $ hence $\m^5=0.$ One of the most significant information on the structure is given by the Hilbert function. 

By definition,  the Hilbert function of $A,$ $h_i=h_i(A) = \dim_K \m^i/\m^{i+1}, $ is the Hilbert function of the associated graded ring $gr_\m(A)=\oplus_{i\geq 0}\m^i/\m^{i+1}.$ We also say that $h=(h_0,h_1,\ldots,h_s)$ is the $h$-vector of $A.$ In \cite{mac27} Macaulay 
characterized the possible sequences of positive integers $h_i $ that can occur as the Hilbert function of $A.$  
Since then there has been a great interest in commutative algebra in determining the $h$-vectors  that can occur as the Hilbert function of $A$ with additional properties (for example, complete intersection, Gorenstein, level, etc). A sequence of positive integers $h=(h_0,h_1,\ldots,h_s)$ satisfying Macaulay's criterion, that   is $h_0=1$ and $h_{i+1} \leq h_i^{\langle i \rangle}$ for $i=1,\ldots,s-1, $ is called an O-sequence.   A sequence $h=(1,h_1,\ldots,h_s)$ is said to be a {\it{level (resp. Gorenstein) O-sequence}}  if $h$ is the Hilbert function of some Artinian level (resp. Gorenstein) $K$-algebra $A.$ Remark that $h_1$ is the embedding dimension and,  if $A$ is level,  $h_s$ is the type   of $A. $     Notice that a level O-sequence is not necessarily the Hilbert function of an Artinian level graded   $K$-algebra. This is because the Hilbert function of the level ring $(A,\m) $ is the Hilbert funtion of $gr_{\m}(A) $ which is not necessarily level. From now on we say that $h$ is a graded level O-sequence if $h$ is the Hilbert function of a level graded standard $K$-algebra.
  For instance  it is well known that the $h$-vector of a  Gorenstein graded $K$-algebra is symmetric, but this is no longer true for a  Gorenstein local ring.  Characterize  level  O-sequences  is a wide open problem in commutative algebra. The problem is difficult and very few is known even in the graded case as evidenced by   \cite{ghms}.  
In the following table we give a summary of known results: 
\vskip 2mm
\begin{center}
\begin{tabular}{|c|c|c|}
\hline
{\bf Characterization of}&{\bf Graded}&{\bf Local}\\
\hline
Gorenstein O-sequences with $h_1=2$&  \cite[Theorem 4.2]{stanley} & \cite[Theorem 2.6]{bertella}\\
\hline
level O-sequences with $h_1=2$& \cite[Theorem 4.6A]{iar84}, \cite{iar04} &\cite[Theorem 2.6]{bertella}\\
\hline
Gorenstein O-sequences with $h_1=3$& \cite[Theorem 4.2]{stanley}& Open\\
\hline
& Open. In \cite{ghms} authors  &\\ 
level O-sequences with $h_1=3$& gave a complete list with &\\
and $\tau \geq 2$&  $s \leq 5$ or $s=6$ and $\tau =2$ &  Open\\
\hline
level O-sequences with $s \leq 3$&  \cite{stefani}, \cite{ghms} & \cite[Theorem 4.3]{stefani}\\
\hline
\end{tabular}
\end{center}
\vskip 2mm
\noindent 
In this paper we fill  the above table by characterizing the   Gorenstein and  level O-sequences with a particular attention to    socle degree $4$ and embedding dimension $h_1=3.$ 

\noindent   In our setting we can write $A=R/I $ where $R=K[[x_1,\dots, x_r]] $ (resp. $K[x_1,\ldots,x_r]$) is the formal power series ring (resp. the polynomial ring with standard grading) and $I$ an ideal of $R.$  We say that $A$ is graded when it can be presented as $R/I$ where $I$ is an homogeneous ideal in $R=K[x_1,\ldots,x_r].$ Without loss of generality we assume that $ h_1= \dim_K \m/\m^2=r.$ \vskip 2mm
\noindent 

Recall that the {\it socle type} of $A=R/I$ is the sequence $E=(0,e_1,\ldots,e_s),$ where 
$$e_i:=\dim_K((0:\m) \cap \m^i) /(0:\m) \cap \m^{i+1}).$$ 
It is known that for all $i \geq 0,$ 
\begin{eqnarray}
\label{eqn:NaturalBound}
 h_i \leq \min\{\dim_K R_i,e_i \dim_K R_{0}+e_{i+1} \dim_K R_1+\cdots+e_s \dim_K R_{s-i}\}
\end{eqnarray}
(see \cite{iar84}). 
Hence a necessary condition for $h$ to be a   level O-sequence is that $h_{s-1}  \leq h_1  h_s$ being $e_s=h_s$ and $e_i=0$ otherwise. In the following theorem we prove that  this condition is also sufficient for $h=(1,3,h_2, h_3, h_4)$  to be a   level O-sequence, provided $h_4  \geq 2.$ However, if $h_4=1, $ we need an additional assumption  for $h$ to be a   Gorenstein O-sequence.  We remark that the result can not be extended  to $h_1>3 $ (see Example \ref{Example:h1=4}). 
\begin{customthm}{1}
 \label{thm:s=4Intro}
Let $h=(1,3,h_2, h_3, h_4)$ be an O-sequence.  
\begin{enumerate}
 \item \label{thm:GoreCaseIntro} Let $h_4 =1.$ Then $h$ is a   Gorenstein O-sequence if and only if $h_3 \leq 3$ and $h_2 \leq \binom{h_3+1}{2}+(3-h_3).$
 \item \label{thm:levelCaseIntro} Let $h_4 \geq 2$. Then $h$ is a   level O-sequence if and only if $h_3 \leq 3 h_4.$
\end{enumerate}
\end{customthm}
The proof of the above result is effective in the sense that in each case  we  construct a local level $K$-algebra with a  given $h$-vector  verifying the necessary conditions (see Theorem \ref{thm:s=4}).

 Combining  Theorem \ref{thm:s=4Intro}\eqref{thm:levelCaseIntro} and results in \cite{ghms} we show that there are   level O-sequences which are not admissible in the graded case  (see Section \ref{section:s=4,c=3}, Table \ref{table:locallevelSeq}). A similar behavior was observed in \cite{stefani} for socle degree   $3.$ Theorem \ref{thm:s=4Intro}\eqref{thm:GoreCaseIntro} is a consequence of the following more general result which holds for any embedding dimension:

\begin{customthm}{2}
\label{thm:ConForGorIntro}
\begin{enumerate}
 \item \label{thm:NecConForGorIntro}
If $(1, h_1,\ldots,h_{s-2}, h_{s-1},1)$ is a   Gorenstein O-sequence, then 
 \begin{equation}
 \label{eqn:NecCondIntro}
 h_{s-1} \leq h_1 \mbox{ and } h_{s-2} \leq  \binom{h_{s-1}+1}{2}+(h_1 - h_{s-1}).
  \end{equation}
  \item \label{thm:SuffFors=4Intro} If $h=(1, h_1, h_2, h_3, 1)$ is an unimodal O-sequence  
  satisfying  \eqref{eqn:NecCondIntro}, then $h$ is a   Gorenstein O-sequence. 
\end{enumerate}
\end{customthm}

The bound \eqref{eqn:NecCondIntro} improves the bound for $h_{s-2}$ given in \eqref{eqn:NaturalBound}.    

\vskip 3mm
If  $A$ is a Gorenstein local $K$-algebra with symmetric $h$-vector, then $gr_{\m}(A) $ is Gorenstein, see \cite{iar94}.   It is a natural question to ask, in this case, whether  $A$ is  analytically isomorphic to $gr_{\m}(A). $  Accordingly with the definition given in \cite[Page 408]{ems78} and in \cite{er12}, recall that an Artinian local $K$-algebra $(A, \m)$ is said to be {\it canonically graded} if there exists a $K$-algebra isomorphism between $A$ and its associated graded ring $gr_{\m}(A).$ 

\vskip 2mm For instance J. Elias and M. E. Rossi  in \cite{er12} proved that every Gorenstein $K$-algebra with  symmetric $h$-vector and $\m^4=0$ ($s \le 3$)  is   canonically graded.   
 A local $K$-algebra $A$ of socle type $E$ is said to be {\it compressed} if equality holds in \eqref{eqn:NaturalBound} for all $1\leq i \leq s, $ equivalently the $h$-vector is maximal (see \cite[Definition 2.3]{iar84}). In \cite[Theorem 3.1]{er15} Elias and Rossi proved that if $A$ is any compressed Gorenstein local $K$-algebra of socle degree $s \le 4,$   then $A$ is canonically graded. In Section \ref{section:CanGra}   we prove that if the socle degree is $4, $ then the assumption can not be relaxed. More precisely only the maximal $h$-vector  forces every corresponding   Gorenstein $K$-algebra to be canonically graded. We prove that if $h$ is unimodal and  not maximal, then there exists a Gorenstein $K$-algebra with Hilbert function $h$ which is not canonically graded.
   To prove that a local $K$-algebra is not canonically graded  is in general a very difficult task. 
See also \cite{joa} for interesting discussions.

\begin{customthm}{3} \label{customthm}
Let $h=(1,h_1,h_2,h_3,1)$ be a     Gorenstein O-sequence with $h_2 \ge h_1.$  Then every local Gorenstein $K$-algebra with  Hilbert function $h$ is necessarily canonically graded if and only if $h_1=h_3$ and $h_2=\binom{h_1+1}{2}.$  
\end{customthm}

The main tool of the paper is  Macaulay's inverse system \cite{mac16} which gives a one-to-one correspondence between ideals $I\subseteq R  $ such that $R/I$ is an Artinian local ring and finitely generated $R$-submodules of a polynomial ring. In Section \ref{section:preliminaries} we gather preliminary results   needed for this purpose.   We prove Theorems \ref{thm:s=4Intro} and \ref{thm:ConForGorIntro} in Section \ref{section:s=4,c=3}, and Theorem \ref{customthm} in Section \ref{section:CanGra}.

We have used Singular \cite{singular}, \cite{Elias} and CoCoA \cite{cocoa} for various computations and examples.

\section*{Acknowledgements}
We thank  Juan Elias for providing us the updated version of {\sc Inverse-syst.lib} for computations.

\section{preliminaries}
\label{section:preliminaries}

\subsection{Macaulay's Inverse System} 
In this subsection we recall some results on Macaulay's inverse system which we will use in the subsequent sections. This theory is well-known in the literature, especially in the graded setting (see for example \cite[Chapter IV]{mac16} and \cite{ik99}). However, the local case is not so well explored. We refer the reader to \cite{ems78}, \cite{iar94} for an extended treatment.  

It is known that the injective hull of $K$ as $R$-module is isomorphic to a divided power ring $P=K[X_1,\ldots,X_r] $ which has a structure of $R$-module by means of the following action
$$
\begin{array}{ cccc}
\circ: & R  \times P  &\longrightarrow &  P   \\
                       &       (f , g) & \to  &  f  \circ g = f (\partial_{X_1}, \dots, \partial_{X_r})(g)
\end{array}
$$
where $  \partial_{X_i} $ denotes the partial derivative with respect to $X_i.$  For the sake of simplicity from now on we will use $x_i $ instead of the capital letters $X_i.$ 
If $\{f_1,\ldots,f_t\} \subseteq P$ is a set of polynomials, we will denote by $\langle f_1,\ldots,f_t \rangle_R$ the $R$-submodule of $P$ generated by $f_1,\ldots,f_t$, i.e., the $K$-vector space generated by $f_1,\ldots,f_t$ and by the corresponding derivatives of all orders. 
We consider the exact pairing of $K$-vector spaces:
$$
\begin{array}{ cccc}
\langle\ , \ \rangle : & R  \times P  &\longrightarrow &   K   \\
                    &       (f , g) & \to  & ( f  \circ g ) (0).
\end{array}
$$
For any ideal  $I\subset R$ we define the following $ R$-submodule of $P $ called Macaulay's inverse system: 
$$
 {I^{\perp}}:=\{g\in P\ |\  \langle f, g \rangle = 0 \ \ \forall f  \in I \ \}.
$$  
Conversely, for every $R$-submodule $M$ of $P$ we define
$$
\ann_R(M):=\{g \in R \ |\  \langle g, f \rangle = 0 \ \ \forall f \in M  \} 
$$  which  is an ideal of $R. $
If $M$ is generated by polynomials $\underbar{f}:=f_1,\ldots,f_t,$ 
with $f_i \in P, $ then we will write $\ann_R(M)= \ann_R(\underbar{f}) $ and  $ A_{\underbar{f}}=R/\ann_R(\underbar{f}).$ 

By using Matlis duality one proves  that there exists a one-to-one correspondence between ideals $I\subseteq R  $ such that $R/I$ is an Artinian local ring and $R$-submodules $M$ of $P$ which are finitely generated. More precisely,  Emsalem in \cite[Proposition 2]{ems78} and Iarrobino in \cite[Lemma 1.2]{iar94}   proved that:

\begin{proposition}  
There is a one-to-one correspondence between ideals $I$ such that $R/I$ is a level local ring of socle degree $s$ and type $\tau$ and $R$-submodules of $P$ generated by $\tau$ polynomials of degree $s$ having linearly independent forms of degree $s.$ The correspondence is defined as follows:
\begin{eqnarray*}
\left\{  \begin{array}{cc} I \subseteq R \mbox{ such that } R/I \\
\mbox{ is a level local ring of }\\
\mbox{ socle degree $s$ and type } \tau 
  \end{array} \right\}  \ &\stackrel{1 - 1}{\longleftrightarrow}& \ 
\left\{ \begin{array}{cc} M \subseteq P \mbox{ submodule generated by} \\
 \tau \mbox{ polynomials of degree} \\
s  \mbox{ with l.i. forms  of degree } s 
\end{array} \right\} \\
I \ \ \ \ \ &\longrightarrow& \ \ \ \ \ I^\bot  \ \ \ \ \ \ \ \ \qquad \ \ \ \  \ \ \ \qquad \ \ \ \  \ \ \ \\
\ann_R(M) \ \ \ \ \  &\longleftarrow& \ \ \ \ \  M \ \ \ \ \ \ \ \ \ \ \ \ \ \ \ \  \qquad \qquad \ \ \ \  \ \ \
\end{eqnarray*}
 \end{proposition}
 
\noindent By \cite[Proposition 2(a)]{ems78}, the action $\langle~ ,~ \rangle $ induces the 
following isomorphism of $K$-vector spaces:
\begin{equation} 
\label{pairing}  (R/I)^* \simeq   {I^{\perp}}, 
\end{equation}
(where $(R/I)^*$ denotes the dual with respect to the pairing $\langle~ ,~ \rangle $ induced on $R/I$). 
Hence $\dim_K R/I  = \dim_K I^{\perp}.$ As in the graded case, it is possible to compute the Hilbert function of $A=R/I $ via the inverse system.
We define the following $K$-vector space:
\begin{equation*} \label{eqn:DualOfI}
 (I^{\perp})_i := {\frac{ I^{\perp} \cap P_{\le i} +  P_{< i}}{ P_{< i}}}.
\end{equation*}
Then, by (\ref{pairing}), it is known that
\begin{equation} \label{H2}
h_i(R/I)= \dim_K    (I^{\perp})_i .
\end{equation}

\subsection{Q-decomposition}
It is well-known that the Hilbert function of an Artinian graded Gorenstein $K$-algebra is symmetric, which is not true in the local case. The problem comes from the fact that, in general, the associated algebra $G:=gr_{\m}(A)$ of a Gorenstein local algebra $A$ is no longer Gorenstein. However, in  \cite{iar94} Iarrobino proved that the Hilbert function of a Gorenstein local $K$-algebra $A$ admits a ``symmetric'' decomposition. To be more precise,   
consider a filtration  of $G $  by a descending sequence of ideals:
$$ G = C(0) \supseteq C(1) \supseteq \dots \supseteq  C(s)=0,  $$
where 
$$C(a)_i:=\frac{(0:\m^{s+1-a-i}) \cap \m^i}{(0:\m^{s+1-a-i})\cap \m^{i+1}}.$$
Let
$$ Q(a)= C(a)/C(a+1).$$
Then 
\begin{equation*}
\{Q(a) : a=0,\ldots,s-1\}
\end{equation*}
is called {\it Q-decomposition} of the associated graded ring $G.$ We have
\begin{eqnarray*}
h_i(A)=\dim_K G_i =\sum_{a=0}^{s-1} \dim_K Q(a)_i. 
\end{eqnarray*} Iarrobino \cite[Theorem 1.5]{iar94} proved that if $A=R/I$ is a Gorenstein local ring then for all $a=0,\ldots,s-1,$ $Q(a)$ is a reflexive graded $G$-module, up to a shift in degree: $\Hom_K(Q(a)_i,K) \cong Q(a)_{s-a-i}.$ Hence the Hilbert function of $Q(a)$ is symmetric about $\frac{s-a}{2}.$ Moreover, he showed that $Q(0) = G/C(1) $ is the unique (up isomorphism) socle degree $s$ graded Gorenstein quotient of $G.$ Let  $f=f[s]+{\text{ lower degree terms...} } $ be a polynomial in $P$ of degree $s$ where  $f[s] $ is the homogeneous part of degree $s$   and consider $A_f$  the corresponding  Gorenstein local $K$-algebra. Then, $Q(0) \cong R/\ann_R(f[s])$ (see \cite[Proposition 7]{ems78} and \cite[Lemma 1.10]{iar94}).

\section{Characterization of   level  O-sequences}
\label{section:s=4,c=3}
In this section we characterize   Gorenstein and level O-sequences of socle degree $4$ and  embedding dimension  $3.$   First if $h=(1,h_1,\ldots,h_s)$ is a   Gorenstein O-sequence (in any embedding dimension),    we give an upper bound on $h_{s-2}$ in terms of $h_{s-1}$ which improves the already known inequality   given in \eqref{eqn:NaturalBound}. 

\begin{theorem}
 \label{thm:ConForGor}
\begin{enumerate}
 \item \label{thm:NecConForGor}
If $(1, h_1,\ldots,h_{s-2}, h_{s-1},1)$ is a   Gorenstein O-sequence, then 
 \begin{equation}
 \label{eqn:NecCond}
 h_{s-1} \leq h_1 \mbox{ and } h_{s-2} \leq \left(\binom{h_{s-1}+1}{2}+(h_1 - h_{s-1})\right).
  \end{equation}
  \item \label{thm:SuffFors=4} If $h=(1, h_1, h_2, h_3, 1)$ is an O-sequence such that 
  $h_2\geq h_3$ and it satisfies \eqref{eqn:NecCond}, then $h$ is a   Gorenstein O-sequence. 
\end{enumerate}
\end{theorem}
\begin{proof}
(\ref{thm:NecConForGor}): From \eqref{eqn:NaturalBound} it follows that $ h_{s-1} \leq h_1.$ 
Let $A$ be a Gorenstein local $K$-algebra with the Hilbert function $h$ and $\{Q(a):a=0,\ldots,s-1\}$ be a $Q$-decomposition of $gr_\m(A).$ Since $Q(i)_{s-1}=0$ for $i>0,$ 
$\dim_K Q(0)_{s-1}=h_{s-1}.$ Hence $\dim_K Q(0)_1=\dim_K Q(0)_{s-1}=h_{s-1}$ and $\dim_K Q(0)_{s-2}=\dim_K Q(0)_2 \leq h_{s-1}^{\langle 1 \rangle}.$ This in turn implies that $\dim_K Q(1)_{s-2}=\dim_K Q(1)_1 \leq h_1 - h_{s-1}.$ Therefore
\begin{eqnarray*}
 h_{s-2}&=& \dim_K Q(0)_{s-2}+\dim_K Q(1)_{s-2}\\
 &\leq& \binom{h_{s-1}+1}{2}+(h_1-h_{s-1}).
\end{eqnarray*}

\eqref{thm:SuffFors=4}: Suppose $h_2\leq h_1.$ Set
  $$
    f=  x_1^4+\cdots+x_{h_3}^4+x_{h_3+1}^3+\cdots+x_{h_2}^3+x_{h_2+1}^2+\cdots+x_{h_1}^2. 
  $$
Then $A_f$ has the Hilbert function $h.$ Now assume $h_2 >h_1.$ Denote  $h_3:=n $ and define monomials $g_i \in K[x_1,\ldots,x_n]$ as follows:
\begin{eqnarray*}
g_i=\begin{cases}
x_i^2 & \mbox{ if } 1 \leq i \leq n\\
x_{i-n}x_{i-n+1} & \mbox{ if } n+1 \leq i \leq 2n-1\\
x_nx_1 & \mbox{ if }i=2n.
\end{cases}
\end{eqnarray*}
For $\ui=(i_1,\ldots,i_n) \in \mathbb{N}^n,$ let $x^{\ui}:=x_1^{i_1} \ldots x_n^{i_n}.$
Let $T$ be the set of monomials $x^{\ui} $ of degree $2$ in $K[x_1,\ldots,x_n]$ such that $$x^{\ui} \notin \{g_i:1\leq i \leq 2n\}. 
$$ Then $|T|=\binom{n+1}{2}-2n.$ We write 
$T=\{g_i: 2n < i \leq \binom{n+1}{2}\}.$ 
Define 
\begin{eqnarray*}
 f=   \begin{cases}\displaystyle
 \sum_{i=1}^{n}x_i^2g_{i}+\sum_{i=1}^{h_2-h_1} x_i^2 g_{n+i} + x_{n+1}^3+\cdots+x_{h_1}^3 & \mbox{ if }h_2 - h_1 \leq n\\
 \displaystyle \sum_{i=1}^{n}x_i^2g_{i} + \sum_{i=1}^{n} x_i^2 g_{n+i} +
 \sum_{2n+1}^{h_2-h_1+n} g_i^2+ x_{n+1}^3+\cdots+x_{h_1}^3 & \mbox{ if }h_2-h_1>n.       
      \end{cases}
 \end{eqnarray*}
Then $h_3=\dim_K\left(\frac{\partial f}{\partial x_1}, \ldots, \frac{\partial f}{\partial x_n}\right)=n$ 
and
\begin{eqnarray*}
h_2=\dim_K \left(\{g_i: 1 \leq i \leq h_2-h_1+n\} \Vx \{x_{n+1}^2,\ldots,x_{h_1}^2\}\right).
\end{eqnarray*}
Thus $A_f$ has the Hilbert function $(1,h_1,h_2,h_3,1).$
\end{proof}
\vskip 2mm

\begin{remark}
\label{remark:gradedGoreCase}
 If $h_3=h_1,$ then $f$ in the proof of Theorem \ref{thm:ConForGor}\eqref{thm:SuffFors=4} is homogeneous and hence $A_f$ is a graded Gorenstein $K$-algebra.
\end{remark}

If the socle degree is $4 $ and $h_1 \le 12, $ then the converse holds in  (\ref{thm:NecConForGor}).

\begin{corollary}
\label{Cor:ConForGor}
 An O-sequence $h=(1, h_1, h_2, h_3, 1)$ with $h_1\leq 12$ is a   Gorenstein O-sequence if and only if $h$ satisfies \eqref{eqn:NecCond}.
\end{corollary}
\begin{proof}
 By Theorem \ref{thm:ConForGor}, it suffices to show that $h_2 \geq h_3$ if $h_1\leq 12.$ Let $A$ be a local Gorenstein $K$-algebra with the Hilbert function $h.$ Considering the symmetric $Q$-decomposition of $gr_{\m}(A),$ we observe that in this case $Q(0)$ has the Hilbert function $(1, h_3,h_2-m, h_3,1),$ for some non-negative integer $m.$  
 Since $h_1\leq 12,$ by \eqref{eqn:NecCond} $h_3 \leq 12. $ As $Q(0)$ is a graded Gorenstein $K$-algebra, by \cite[Theorem 3.2]{mz} we conclude that $h_2-m \geq h_3 $ since $Q(0) $ has unimodal Hilbert function, hence $h_2 \ge h_3.$  
\end{proof}

\begin{remark}
\label{remark:UnimodalGore}
\label{remark:NonuniGore}
A   Gorenstein O-sequence $(1, h_1, h_2, h_3,1)$ does not necessarily satisfy $h_2  \geq h_3.$ For example, consider the the sequence $h=(1,13,12,13,1).$ By \cite[Example 4.3]{stanley} there exists a graded Gorenstein $K$-algebra with $h$ as $h$-vector. 
\end{remark}

In the following theorem we  characterize  the $h$-vector of local level algebras of socle degree $4 $  and embedding dimension $3.$ 

\begin{theorem}
\label{thm:s=4}
Let $h=(1,3,h_2, h_3, h_4)$ be an O-sequence.  
\begin{enumerate}
 \item \label{thm:GoreCase} Let $h_4 =1.$ Then $h$ is a   Gorenstein O-sequence if and only if $h_3 \leq 3$ and $h_2 \leq \binom{h_3+1}{2}+(3-h_3).$
 \item \label{thm:levelCase} Let $h_4 \geq 2$. Then $h$ is a   level O-sequence if and only if $h_3 \leq 3 h_4.$
\end{enumerate}
 \end{theorem}
\begin{proof}
\eqref{thm:GoreCase}: 
Follows from Corollary \ref{Cor:ConForGor}.

\noindent \eqref{thm:levelCase}: 
The ``only if'' part follows from \eqref{eqn:NaturalBound}. The converse is constructive and we prove it   by inductive steps on $h_4.$  First we consider the cases $h_4=2,3,4$ and then $h_4 \geq 5.$ In each case we define the polynomials $\underbar{f}:=f_1,\ldots,f_{h_4} \in P= K[x_1,x_2,x_3] $ of degree $4$ such that $A_{\underbar{f}}$ has the Hilbert function $h.$ We set $g_1'=x_3^3,g_2'=x_2^2x_3,g_3'=x_1^2x_2,g_4'=x_1x_3^2.$
\vskip 2mm
\noindent 
For short, in this proof we use the following notation: $m:=h_2$ and $n:=h_3.$

\begin{center}
\noindent{\underline{{\bf Case 1: $ h_4=2.$}}}
 \end{center}
In this case $n \leq 6$ as $h_3 \leq 3 h_4$ by assumption. Suppose $m=2.$ Then $h$ is an O-sequence implies that $n=2.$ In this case, let $f_1=x_1^4+x_3^2$ and $f_2=x_2^4.$ Then $A_{\underbar{f}}$ has the Hilbert function $(1,3,2,2,2).$ Now assume that $m \geq 3.$  \\
{\underline{{\bf Subase 1: $m \geq n.$}}} 
We set $g_i=\begin{cases}
          x_{4-i}g_i' & \mbox{ if } 1 \leq i \leq n-2\\
          g_i' & \mbox{ if }n-1 \leq i \leq m-2\\
          0   & \mbox{ if } m-1 \leq i \leq 4.
          \end{cases}$      
              
\noindent (Here $x_0=x_3$). Define
\begin{equation*}
f_1=x_1^4+g_1+g_2 \mbox{ and }f_2=x_2^4+g_3+g_4. 
\end{equation*}
Then 
\begin{eqnarray*}
 h_3&=&\dim_K\{x_1^3,x_2^3,g_1',\ldots,g_{n-2}'\}=n \mbox{ and }\\
 h_2&=&\dim_K\{x_1^2,x_2^2,\frac{g_i'}{x_{4-i}}:1\leq i \leq m-2\}=m
\end{eqnarray*}
and hence $A_{\underbar{f}}$ has the required Hilbert function $h.$\\
{\underline{{\bf Subcase 2: $m <n.$}}} 
The only possible ordered tuples $(m,n)$ with $m<n \leq 6$ such that $h$ is an O-sequence are $\{(3,4),(4,5),(5,6)\}.$ For each $2$-tuple $(m,n)$ we define $f_1,f_2$ as:\\
$a. (m,n)=(3,4):$ $f_1=x_1^4+x_1^2x_2^2+x_3^2;f_2=x_2^4+x_1^2x_2^2.$ \\
$b. (m,n)=(4,5):$ $f_1=x_1^4+x_1^2x_2^2+x_3^4;f_2=x_2^4+x_1^2x_2^2.$\\
$c. (m,n)=(5,6):$ $f_1=x_1^4+x_1^2x_2^2+x_3^4;f_2=x_2^4+x_1^2x_2^2+x_2^3x_3.$
\begin{center}
\noindent{\underline{{\bf Case 2: $ h_4=3.$}}} 
\end{center}
In this case $n \leq 9.$ 
We consider the following subcases:\\
{\underline{{\bf Subcase 1: $n \leq 6.$}}} Let $\underbar{f}^\prime=f_1,f_2$ be polynomials defined as in Case 1 such that $A_{\underbar{f}^\prime}$ has the Hilbert function $(1,3,m,n,2).$ Now define 
$f_3=\begin{cases}
      x_3^4 &\mbox{ if } m\geq n\\
      x_1^2x_2^2 &\mbox{ if }m<n.     
     \end{cases}
$\\
Then $A_{\underbar{f}}$ has the required Hilbert function $h.$\\
{\underline{{\bf Subcase 2: $7 \leq n \leq 9.$}}} Let $\underbar{f}^\prime=f_1,f_2$ be polynomials defined as in Case 1 such that $A_{\underbar{f}^\prime}$ has the Hilbert function $(1,3,m,6,2).$ 
We set $p_1=x_2^2x_3^2,p_2=x_1^2x_2^2$ and $p_3=x_1^2x_3^2.$ 
Since $h$ is an O-sequence and $n\geq 7$, we get $ m\geq 5.$ 
Now define 
$
f_3=\begin{cases}
     \sum_{i=1}^{n-6} p_i &\mbox{ if } m=6\\
     x_2^2x_3^2 & \mbox{ if }m=5.
     \end{cases}
$\\
Then $A_{\underbar{f}}$ has the required Hilbert function $h.$
 \begin{center}
\noindent{\underline{{\bf Case 3: $h_4 =4.$}}} 
\end{center}
Since $h$ is an O-sequence, $n \leq 10.$ 
We consider the following subcases:\\
\noindent {\underline{{\bf Subcase 1: $n \leq 9.$}}} Let $\underbar{f}^\prime=f_1,f_2,f_3$ be polynomials defined as in Case 2 such that $A_{\underbar{f}^\prime}$ has the Hilbert function $(1,3,m,n,3).$ Define
$$
f_4=\begin{cases}
     x_2^3x_3  &  \mbox{ if } \{m \geq n \mbox{ and } n \leq 6\} \mbox{ OR } \{ n \geq 7 \mbox{ and }m=6\}\\
     x_1^3x_2  & \mbox{ if }\{m<n \leq 6\} \mbox{ OR }\{(m,n)=(5,7)\}.
     \end{cases}
$$
Then $A_{\underbar{f}}$ has the Hilbert function $(1,3,m,n,4).$\\
{\underline{{\bf Subcase 2: $n=10.$}}} As $h$ is an O-sequence, we conclude that $m=6.$ Let $\underbar{f}^\prime=f_1,f_2,f_3$ be polynomials defined as in Case 2 such that $A_{\underbar{f}^\prime}$ has the Hilbert function $(1,3,6,9,3).$ Define $f_4=x_1^2x_2x_3.$
Then $A_{\underbar{f}}$ has the Hilbert function $(1,3,6,10,4).$

\begin{center}
\noindent{\underline{{\bf Case 4: $ h_4 \geq 5.$}}} 
\end{center}
Since $h$ is an O-sequence, $n \leq 10$ and $h_4 \leq 15$ \\
{\underline{{\bf Subcase 1: $n \geq h_4$ OR $h_4 \geq 11.$ }}}
Let $\underbar{f}^\prime=f_1,f_2,f_3,f_4$ be defined as in Case 3 such that  $ A_{\underbar{f}^\prime} $ has the Hilbert function  $(1,3,m,n,4).$ For $5 \leq i \leq 15,$ define $f_i$ as follows:
\begin{eqnarray*}
f_5&=&\begin{cases}
     x_1^3x_2  &  \mbox{ if } \{m \geq n \mbox{ and } n \leq 6\} \mbox{ OR }  \{n \geq 7 \mbox{ and }m=6\}\\
     x_1x_2^3  & \mbox{ if }\{m<n \leq 6\} \mbox{ OR }\{(m,n)=(5,7)\},
     \end{cases} \\
f_6&=&\begin{cases}
     x_1x_3^3  &  \mbox{ if } \{m \geq n \mbox{ and } n \leq 6\} \mbox{ OR }  \{n \geq 7 \mbox{ and }m=6\}\\
     x_3^4  & \mbox{ if }\{m<n \leq 6 \}\mbox{ OR }\{(m,n)=(5,7)\},
     \end{cases} \\
f_7&=&\begin{cases}
     x_3^4  &  \mbox{ if }  n \geq 7 \mbox{ and }m=6\\
     x_2^3x_3  & \mbox{ if }\{(m,n)=(5,7)\}.
     \end{cases} 
\end{eqnarray*}
(Note that in the last case, $h_4 \geq 7$ implies that $n \geq 7$). 
If $h_4 \geq 8,$ then $n \geq 8$ which implies that $m=6.$ We set
$$f_8=x_1^2x_2^2,f_9=x_1^2x_3^2,f_{10}=x_2x_3^3,f_{11}=x_1x_2^3,f_{12}=x_1^3x_3,f_{13}=x_2^3x_3,f_{14}=x_1x_2^2x_3,f_{15}=x_1x_2x_3^2.$$ Now $A_{\underbar{f}}$ has the Hilbert function $(1,3,m,n,h_4).$

\noindent {\underline{{\bf Subcase 2: $n < h_4 \leq 10.$ }}} The smallest ordered tuple $(n, h_4)$ such that $h$ is an O-sequence and $n<h_4$ is $(4,5).$ (Here smallest ordered tuple means smallest with respect to the order $\leq$ defined as: $(n_1,n_2) \leq (m_1,m_2)$ if and only if $n_1\leq m_1$ and $n_2\leq m_2 $). Let  
\begin{eqnarray*}
q_1=\begin{cases}
          x_3^2 & \mbox{ if } m=3\\
          x_3^3 & \mbox{ if } m \geq 4,
          \end{cases}
q_2=\begin{cases}
          0   & \mbox{ if } m< 5\\
          x_2^2x_3 & \mbox{ if } m\geq 5
     \end{cases} 
\mbox{ and }     
q_3=\begin{cases}
          0   & \mbox{ if } m< 6\\
          x_1x_3^2 & \mbox{ if } m= 6.
     \end{cases} 
     \end{eqnarray*}
We define 
\begin{equation*}
\label{eqn:basecase}
f_1=x_1^4+q_1+q_2,f_2=x_2^4+q_3,f_3=x_1^3x_2,f_4=x_1x_2^3,f_5=x_1^2x_2^2.
\end{equation*}
Then $A_{\underbar{f}}$ has the Hilbert function $(1,3,m,4,5).$

\noindent Let $h_4 \geq 6.$ We set 
$$f_6=x_3^4,f_7=x_2^3x_3,f_8=x_2x_3^3,f_9=\begin{cases}
                                                                    x_2^2x_3^2 &\mbox{ if }n=7\\
                                                                    x_1x_3^3 &\mbox{ if }n\geq8,
                                                                   \end{cases}
f_{10}=\begin{cases}
         x_2^2x_3^2 &\mbox{ if }n=8\\
         x_1^3x_3 &\mbox{ if } n=9.     
         \end{cases}
$$ 
Then $A_{\underbar{f}}$ has the Hilbert function $(1,3,m,n,h_4).$
\end{proof}

Using \cite[Appendix D]{ghms} and Theorem \ref{thm:s=4}\eqref{thm:levelCase} we list in the following table all the O-sequences which are realizable for local level $K$-algebras with $h_4 \ge 2$, but not for graded level $K$-algebras.

\begin{table}[h]
\caption{}
\label{table:locallevelSeq}
\begin{center}
 \begin{tabular}{|c|c|c|c|c|c|}
\hline
$(1,3,2,2,2)$ & $(1,3,3,2,2)$ &$(1,3,4,2,2)$ & $(1,3,5,2,2)$ & $(1,3,6,2,2)$ & $(1,3,5,3,2)$\\
\hline
$(1,3,6,3,2)$ &$(1,3,3,4,2)$ & $(1,3,4,3,3)$ & $(1,3,5,3,3)$ &$(1,3,6,3,3)$& $(1,3,3,4,3)$ \\ 
\hline
 $(1,3,6,4,3)$ & $(1,3,3,4,4)$ & $(1,3,5,4,4)$ & $(1,3,6,4,4)$ &$(1,3,3,4,5)$ & $(1,3,4,4,5)$\\
\hline
$(1,3,5,4,5)$&$(1,3,6,4,5)$& $(1,3,6,5,5)$& $(1,3,5,5,6)$ & $(1,3,6,5,6)$ & $(1,3,6,6,7)$\\
\hline
$(1,3,6,7,9)$&  & &  &  & \\
\hline
\end{tabular}
\end{center}
\end{table}

Analogously,  by Theorem \ref{thm:s=4}\eqref{thm:GoreCase}, the following O-sequences are   Gorenstein sequences, but they are not graded Gorenstein sequences  since they are not symmetric. If an O-sequence $h$ is symmetric with $h_1=3$ then $h$ is also a graded Gorenstein O-sequence by Remark \ref{remark:gradedGoreCase}. 
\begin{table}[h]
\caption{}
\label{table:localGorSeq}
\begin{center}
 \begin{tabular}{|c|c|c|c|c|c|}
 \hline
  $(1,3,1,1,1)$ & $(1,3,2,1,1)$ & $(1,3,3,1,1)$ & $(1,3,2,2,1)$ & $(1,3,3,2,1)$ & $(1,3,4,2,1)$\\
\hline
 \end{tabular}
 \end{center}
 \end{table}

The following example shows  that   Theorem \ref{thm:s=4}\eqref{thm:levelCase} can not be extended to $h_1 \ge 4 $ because the necessary condition $h_3 \le h_1 h_s $ is not longer sufficient for characterizing level O-sequences of  socle degree $4. $ 

\begin{example}
\label{Example:h1=4}
 The O-sequence $h=(1,4,9,2,2)$ is not a   level O-sequence.
\end{example}
\begin{proof}
Let $A=R/I$ be a local level $K$-algebra with the Hilbert function $h.$ The lex-ideal $L \in P=K[x_1,\dots, x_4] $ with the Hilbert function $h $ is $$L=(x_1^2,x_1x_2^2,x_1x_2x_3,x_1x_2x_4,x_1x_3^2,x_1x_3x_4,x_1x_4^2,x_2^3,x_2^2x_3,x_2^2x_4,x_2x_3^2,x_2x_3x_4,x_2x_4^2,x_3^3,x_3^2x_4,x_3x_4^4,x_4^5).$$ 
\noindent  A minimal graded $P$-free resolution of $P/L$ is:
{\small{  $$ 0  \to P(-6)^7 \oplus P(-8)^2   \to P(-5)^{26}\oplus  P(-7)^6   \to P(-4)^{33} \oplus P(-6)^6  \to P(-2) \oplus P(-3)^{14} \oplus P(-5)^2.$$ }}
\noindent By \cite[Theorem 4.1]{rs10} the Betti numbers of $A$ can be obtained from the Betti numbers of $P/L$ by a sequence of negative and zero consecutive cancellations. 
This implies that $\beta_4(A) \geq 3$ and hence $A$ has type at least $3,$ which leads to a contradiction. 
\end{proof} 
 
\section{Canonically graded algebras with $\m^5=0.$}
\label{section:CanGra}
It is clear that 
a necessary condition for a  Gorenstein local $K$-algebra $A$ being  canonically graded is that the Hilbert function of $A$ must be  symmetric. Hence we investigate whether a Gorenstein $K$-algebra $A$ with the Hilbert function $(1, h_1, h_2, h_1, 1)$ is necessarily canonically graded. If $h_2=\binom{h_1+1}{2}$ (equiv. $A$ is compressed), 
then by \cite[Theorem 3.1]{er15} $A$ is canonically graded. In this section we prove that if $h=(1, h_1, h_2, h_1, 1) $ is an O-sequence with $h_1 \le h_2< \binom{h_1+1}{2},$ then there exists a polynomial $F$ of degree $4$ such that $A_F$ has the  Hilbert function $h$ and it is {\it not} canonically graded. 

\begin{theorem}
\label{thm:CanGra}
Let $h=(1,h_1,h_2,h_3,1)$ be a    Gorenstein O-sequence with $h_2 \ge h_1$. Then every   local Gorenstein $K$-algebra with   Hilbert function $ h $  is necessarily canonically graded if and only if $h_1=h_3$ and $h_2 =\binom{h_1+1}{2}.$ 
\end{theorem}
\begin{proof}
The assertion is clear for $h_1=1.$ Hence we assume $h_1 > 1.$ The ``if'' part of the theorem follows from \cite[Theorem 3.1]{er15}. We   prove the converse, that is we show  that     if $h_2 < \binom{h_1+1}{2},$ then there exists a polynomial $G$ of degree $4$ such that $A_G$ has the Hilbert function $h$ and it is not canonically graded. We may  assume $h_3=h_1,  $ otherwise the result is clear. For simplicity in the notation we put $h_1:=n $ and $h_2:=m.$

First we prove the assertion for $n \leq 3.$  
We define
\begin{eqnarray*}
 F=\begin{cases}
    x_1^3x_2 &\mbox{ if }n=m=2\\ 
    x_1^4+x_2^4+x_2^3x_3 &  \mbox{ if }n=3  \mbox{ and } m=3\\
    x_1^4+x_2^4+x_2^3x_3+x_1^3x_2 &  \mbox{ if }n=3  \mbox{ and } m=4\\
    x_1^4+x_2^4+x_2^3x_3+x_1^3x_2+x_1x_2^2x_3 &\mbox{ if }n=3 \mbox{ and }m=5.
    \end{cases}
\end{eqnarray*}
Let $G=F+x_{n}^3.$ It is easy to check that $A_G$ has the Hilbert function $h.$ We claim that $A_G$ is not canonically graded. Suppose that $A_G$ is canonically graded. Then $ A_F \cong gr_{\m}(A)  \cong  A_G.$ Let $\varphi:A_F \to A_G$ be a $K$-algebra automorphism. Since $x_n^2 \circ F=0,$ $x_n^2 \in \ann_R(F).$ This implies that $\varphi(x_n)^2 \in \ann_R(G)$ and hence $\varphi(x_n)^2 \circ G=0.$ For $\ui=(i_1,\ldots,i_n) \in \mathbb{N}^n,$ let $|\ui|=i_1+\cdots+i_n.$ Suppose 
\begin{eqnarray*}
 \varphi(x_n)=u_1x_1+\cdots+u_nx_n+\sum_{\ui \in \mathbb{N}^n,~|\ui|\geq 2} a_{\ui} x^{\ui}.
\end{eqnarray*}
Comparing the coefficients of the monomials of degree $\le 2$   in $\varphi(x_n)^2 \circ G=0,$ it is easy to verify that $u_1=\cdots=u_n=0.$ This implies that $\varphi(x_n)$ has no linear terms and thus $\varphi$ is not an automorphisms, a contradiction.

Suppose $n>3.$ First we define a homogeneous polynomial $F \in P$ of degree $4$ such that $A_F$ has the Hilbert function $h$ and $x_n^2$ does not divide any monomial in $F $ (in other words, if $x^{\ui}$ is a monomial that occurs in $F$ with nonzero coefficient, then $i_n \leq 1$). 

Let $T$ be a  monomials basis of $P_2. $ We split  the set $T\setminus \{x_n^2\}$ into a disjoint union of   monomials as follows. 
We set \begin{eqnarray*}
     p_i=\begin{cases}
      x_i^2 & \mbox{ for } 1 \leq i \leq n-1\\
      x_2x_n & \mbox{ for }i=n\\
      x_{i-n}x_{i+1-n} &\mbox{ for } n+1 \leq i < 2n \\
      x_1x_n & \mbox{ for }i=2n.
      \end{cases}
      \end{eqnarray*}
Let $E=\{p_i:1\leq i \leq n\}, B=\{p_i:n+1 \leq i \leq2n\}, C=\{x_ix_j:1\leq i <j < n \mbox{ such that } j-i>1\} $ and $D:=\{x_ix_n:3\leq i \leq n-2\}.$  Then 
$$T\setminus \{x_n^2\}=E\Vx B \Vx C \Vx D.$$
Denote by $ | \cdot | $ the cardinality, then  $|C|=\binom{n+1}{2}-2n-(n-4)-1$ and $|D|=n-4.$ Hence we write $C=\{p_i: 2n <i\leq \binom{n+1}{2}-(n-4)-1\} $ and $D=\{p_i:\binom{n+1}{2}-(n-4)-1<i \leq \binom{n+1}{2}-1\}.$
We set 
\begin{eqnarray*}
 g_i= \begin{cases}
         x_i^4 & \mbox{ for } 1 \leq i \leq n-1\\
         x_2^3x_n &\mbox{ for } i=n\\
         x_{i-n}^2p_i &\mbox{ for } n+1 \leq i <2n \mbox{ and } i \neq n+2 \\
         x_2^2x_3^2 &\mbox{ for } i=n+2\\
         x_1x_2^2x_n &\mbox{ for } i=2n\\
         p_i^2 &\mbox{ for } 2n <i \leq \binom{n+1}{2} - (n-4)-1\\
         \frac{x_2}{x_n} p_i^2 & \mbox{ for } \binom{n+1}{2} - (n-4) -1 <i <\binom{n+1}{2}.
        \end{cases}
\end{eqnarray*}
Define
\begin{eqnarray*}
 F=\sum_{i=1}^{m} g_i.
\end{eqnarray*}
Since $m \geq n,$ $\dim_K (\langle F\rangle_R)_i=n$ for $i=1,3.$ Also, $\dim_K(\langle F\rangle_R)_2=\dim_K\{p_i:1 \leq i \leq m\}=m.$ Hence $A_F$ has the Hilbert function $h.$  

Let $G=F+x_n^3.$ We prove that $A_G$ is not canonically graded. Suppose that $A_G$ is canonically graded. Then, as before,  $A_G \cong A_F. $ Let $\varphi:A_F \to A_G$ be a $K$-algebra automorphism. Since $F$ does not contain a monomial multiple of  $x_n^2,$ $x_n^2  \circ F=0$ and hence $x_n^2 \in \ann_R(F)$ which implies that $\varphi(x_n)^2 \in \ann_R(G).$ Let 
\begin{eqnarray*}
 \varphi(x_n)=u_{1}x_1+\cdots+u_{n}x_n+\sum_{{\ui} \in \mathbb{N}^n,~|\ui|\geq 2} a_{\ui} x^{\ui}.
 \end{eqnarray*}
We claim that $u_1=\cdots=u_{n-1}=0.$\\
\underline{{\bf Case 1:} $m=n.$} Comparing the coefficients of $x_1^2,x_2x_n,x_3^2,\ldots,x_{n-1}^2$ in $\varphi(x_n)^2 \circ G=0,$ we get $u_{1}=\cdots=u_{n-1}=0.$ \\
\underline{{\bf Case 2:} $m=n+1$ OR $m = n+2.$} Compare the coefficients of $x_1x_2,x_2x_n,x_3^2,\ldots,x_{n-1}^2$ in $\varphi(x_n)^2 \circ G=0,$ to get $u_1=\cdots=u_{n-1}=0.$ \\
\underline{{\bf Case 3:} $n+2<m < 2n.$} Comparing the coefficients of $x_1x_2, x_{2}x_n,x_3x_4, \ldots,x_{m-n}x_{m-n+1},x_{m-n+1}^2,$ $x_{m-n+2}^2\ldots,x_{n-1}^2$ in $\varphi(x_n)^2 \circ G=0,$ we get   $u_{1}=\cdots=u_{n-1}=0.$ \\
\underline{{\bf Case 4:} $m \geq 2n.$} Comparing the coefficients of 
$x_1x_2, x_{1}x_n,x_3x_4, \ldots,x_{n-1}x_{n}$ 
in $\varphi(x_n)^2 \circ G=0,$ we get $u_{1}=\cdots=u_{n-1}=0.$\\ 
This proves the claim. Now, comparing the coefficients of $x_n$ in $\varphi(x_n)^2 \circ G=0,$ we get $u_{n}=0$ (since $F$ does not contain a monomial divisible by $x_n^2$). This implies that $\varphi(x_n) $ has no linear terms and hence $\varphi$ is not an automorphism, a contradiction.
\end{proof}

We expect that the Theorem \ref{thm:CanGra} holds true without the assumption $h_2 \ge h_1$.  The problem is that, as far as we know,  the admissible Gorenstein not unimodal $h$-vectors are not classified even if $s=4.$   However, starting from an example by Stanley, 
we are able to construct a not canonically graded Gorenstein $K$-algebra with (not unimodal)  $h$-vector 
$(1,13,12,13,1).$  

\begin{corollary}
\label{Cor:CanGra}
Let $h=(1, h_1, h_2, h_1, 1)$ where $h_1 \leq 13 $ be a  Gorenstein O-sequence. Then every Gorenstein $K$-algebra with the Hilbert function $h$ is necessarily canonically graded if and only if $h_2 =\binom{h_1+1}{2}.$ 
\end{corollary}
\begin{proof}
If a local Gorenstein $K$-algebra $A$ has  Hilbert function  $h=(1, h_1, h_2, h_1, 1), $ then by considering Q-decomposition of $gr_\m(A)$ we conclude that $gr_\m(A)\cong Q(0).$ This implies that $h$ is also the 
Hilbert function of a graded Gorenstein $K$-algebra. By \cite[Theorem 3.2]{mz} if $h_1\leq 12,$ then the Hilbert function of a graded Gorenstein $K$-algebra is unimodal. Hence by Theorem \ref{thm:CanGra} the result follows.

If $h_1=13$ and $h$ is unimodal, then the assertion follows from Theorem \ref{thm:CanGra}. Now, by \cite[Theorem 3.2]{mz} the only nonunimodal graded Gorenstein O-sequence with $h_1=13$ is $h=(1,13,12,13,1).$ 
In this case we write $P=[x_1,\ldots,x_{10},x,y,z].$ Let  
\begin{eqnarray*}
F=\sum_{i=1}^{10}x_i \mu_i,
\end{eqnarray*}
where $ \mu=\{x^3,x^2y,x^2z,xy^2,xyz,xz^2,y^3,y^2z,yz^2,z^3\}=\{\mu_1,\ldots,\mu_{10}\}.$ 
Let $G=F+x_1^3+\cdots+x_{10}^3.$ 
Then $A_{G}$ has the Hilbert function $h.$ 
We claim that $A_G$ is not canonically graded. Suppose $A_G$ is canonically graded. Then $A_G \cong A_F.$ Let $\varphi:A_F \to A_G$ be a $K$-algebra automorphism. Since $x_1^2 \in \ann_R(F),$ $\varphi(x_1)^2 \in  \ann_R(G).$ Let 
$$\varphi(x_1)=u_1x_1+\cdots+u_{10}x_{10}+u_{11}x+u_{12}y+u_{13}z+\sum_{\ui \in \mathbb{N}^n,~|\ui| \geq 2} a_{\ui}x^{\ui}.$$ 
Comparing the coefficients of $x_1x,x_7y,x_{10}z$ in $\varphi(x_{1})^2 \circ G=0,$ we get $u_{11}=u_{12}=u_{13}=0.$ Now, comparing the coefficients of $x_1,\ldots,x_{10}$ in $\varphi(x_{1})^2 \circ G=0,$ we get $u_1=\cdots=u_{10}=0.$ This implies that $\varphi(x_1)$ has no linear terms and thus $\varphi$ is not an automorphism, a contradiction.
\end{proof}

\newcommand{\etalchar}[1]{$^{#1}$}
\def\cfudot#1{\ifmmode\setbox7\hbox{$\accent"5E#1$}\else
  \setbox7\hbox{\accent"5E#1}\penalty 10000\relax\fi\raise 1\ht7
  \hbox{\raise.1ex\hbox to 1\wd7{\hss.\hss}}\penalty 10000 \hskip-1\wd7\penalty
  10000\box7}

\end{document}